\pdfoutput=1
\RequirePackage{ifpdf}
\ifpdf 
\documentclass[pdftex]{sigma}
\else
\documentclass{sigma}
\fi

\newtheorem{Theorem}{Theorem}[section]

\newtheorem{Lemma}[Theorem]{Lemma}
\newtheorem{Proposition}[Theorem]{Proposition}
 { \theoremstyle{definition}
\newtheorem{Definition}[Theorem]{Definition}

\newtheorem{Example}[Theorem]{Example}
\newtheorem{Remark}[Theorem]{Remark}
\newtheorem{Conjecture}[Theorem]{Conjecture}}

\numberwithin{equation}{section}

\newcommand{\N}{\mathbb{N}}   
\renewcommand{\epsilon}{\varepsilon}    
\newcommand{\calP}{\mathcal{P}}

\begin{document}
\allowdisplaybreaks

\newcommand{\arXivNumber}{2208.11297}

\renewcommand{\PaperNumber}{004}

\FirstPageHeading

\ShortArticleName{Law of Large Numbers for Roots of Finite Free Multiplicative Convolution}

\ArticleName{Law of Large Numbers for Roots of Finite Free\\ Multiplicative Convolution of Polynomials}

\Author{Katsunori FUJIE~$^{\rm a}$ and Yuki UEDA~$^{\rm b}$}

\AuthorNameForHeading{K.~Fujie and Y.~Ueda}

\Address{$^{\rm a)}$~Department of Mathematics, Hokkaido University,\\
\hphantom{$^{\rm a)}$}~North 10 West 8, Kita-Ku, Sapporo, Hokkaido, 060-0810, Japan}
\EmailD{\href{mailto:kfujie@eis.hokudai.ac.jp}{kfujie@eis.hokudai.ac.jp}}

\Address{$^{\rm b)}$~Department of Mathematics, Hokkaido University of Education,\\
\hphantom{$^{\rm b)}$}~Hokumon-cho 9, Asahikawa, Hokkaido, 070-8621, Japan}
\EmailD{\href{mailto:ueda.yuki@a.hokkyodai.ac.jp}{ueda.yuki@a.hokkyodai.ac.jp}}

\ArticleDates{Received August 25, 2022, in final form January 09, 2023; Published online January 14, 2023}

\Abstract{We provide the law of large numbers for roots of finite free multiplicative convolution of polynomials which have only non-negative real roots. Moreover, we study the empirical root distributions of limit polynomials obtained through the law of large numbers of finite free multiplicative convolution when their degree tends to infinity.}


\Classification{46L54; 26C10; 60F05}

\section{Introduction}
\subsection{Free probability theory}

Denote by $\calP$ and $\calP_+$ the set of all probability measures on $\mathbb{R}$ and $[0,\infty)$, respectively.
Moreover, we define $\calP_c$ and $\calP_{+,c}$ as the set of all compactly supported probability measures on $\mathbb{R}$ and $[0,\infty)$, respectively.
The notation $\xrightarrow{w}$ means the weak convergence of sequences of probability measures.

Voiculescu initiated free probability theory to attack problems related to the free product of operator algebras.
One of the most important notions in this theory is the {\it free independence} of non-commutative random variables.
In this paper, we say \textit{free random variables} as freely independent non-commutative random variables for short.

The {\it law of large numbers} (LLN) is well-known as a result that a sample average of independent identically distributed random variables with finite mean concentrates on the theoretical mean when the sample size is sufficiently large.
As the analogous result on classical probability, the LLN for free random variables was also established (see \cite{LP97}).
More precisely, for any $\mu\in\calP$ with mean $\alpha$, we have $D_{1/n}\big(\mu^{\boxplus n}\big)\xrightarrow{w}\delta_\alpha$ as $n\rightarrow\infty$, where (i) $D_{c}(\nu)$ is the push-forward of a~measure~$\nu$ by the mapping $x\mapsto cx$ for $c\in \mathbb{R}$ and (ii) $\mu\boxplus\nu$ is called the {\it free additive convolution}, which is the probability distribution of addition $X+Y$ of free random variables $X$ and $Y$ distributed as $\mu\in \calP$ and $\nu\in \calP$, respectively, in particular $\mu^{\boxplus n}$ is the $n$-th power of free additive convolution of $\mu$.

The LLN for multiplication of (classically or freely) independent positive random variables are also considered.
In classical probability, it is easy to formulate and investigate the LLN of multiplication by considering the exponential mapping of those random variables.
However, it is not easy to consider the LLN for multiplication in free probability since ${\rm e}^{X+Y} \neq {\rm e}^X {\rm e}^Y$ for (non-commutative) random variables $X$ and $Y$. In \cite{Tucci10}, the LLN for multiplication of free bounded positive random variables was obtained. After that, this LLN was extended to one for multiplication of free positive random variables (which are not necessary to be bounded) in \cite{HM13}. More precisely, the LLN for multiplication of free positive random variables can be formulated as the convergence of
\begin{align}\label{eq:LLN_freemultiplicative}
\Big\{ \big(\mu^{\boxtimes n}\big)^{\frac{1}{n}}  \Big\}_{n\in \mathbb{N}},
\end{align}
for $\mu \in \calP_+$, where (i) $\nu^\alpha$ denotes the push forward of a measure $\nu$ by the mapping $x\mapsto x^\alpha$ for $\alpha\in\mathbb{R}$ and (ii) $\mu\boxtimes\nu$ is called the {\it free multiplicative convolution}, which is the probability distribution of multiplication $\sqrt{X}Y\sqrt{X}$ of free random variables $X\ge 0$ and $Y$ distributed as $\mu \in \calP_+$ and $\nu \in \calP$, respectively,  in particular $\mu^{\boxtimes n}$ is the $n$-th power of free multiplicative convolution of $\mu$ (see \cite{Voi87} and \cite{BV93} for more details).
According to \cite{HM13}, the limit distribution of the sequence \eqref{eq:LLN_freemultiplicative} always exists and is denoted by $\Phi(\mu)$. For $\mu\neq \delta_0$, the measure $\Phi(\mu) \in \calP_+$ is characterized by the S-transform (see Section~\ref{section2.1} for details).

\subsection{Finite free probability and main result}

In \cite{Marcus} and \cite{MSS22}, Marcus, Spielman and Srivastava investigated a link between polynomial convolutions and the sum of random matrices related to free probability theory. For monic polynomials $p(x)=\sum_{i=0}^d (-1)^i p_i x^{d-i}$ and $q(x)=\sum_{i=0}^d (-1)^i q_i x^{d-i}$ of degree $d$, the {\it finite free additive convolution} $p\boxplus_d q$ is defined by
\begin{align*}
\big(p\boxplus_d q\big) (x):=\sum_{i+j\le d}(-1)^{i+j} \frac{(d-i)!(d-j)!}{(d-i-j)!d!}p_iq_j x^{d-i-j}.
\end{align*}
The finite free additive convolution plays an important role in studying characteristic polynomials of the sum of (random) matrices. More precisely, for $d\times d$ real symmetric matrices~$A$ and~$B$ with characteristic polynomials $\chi_A$ and $\chi_B$, respectively, $\chi_A\boxplus_d \chi_B$ is given by
\begin{align*}
\big(\chi_A\boxplus_d \chi_B\big)(x)=\mathbb{E}_Q \det[xI_d- A- QBQ^*],
\end{align*}
where the expectation is taken over unitary matrices $Q$ distributed uniformly on the unitary group in dimension $d$.
Furthermore, the finite free additive convolution is very closely related to free additive convolution because it turned out to be that the empirical root distribution of $p_d\boxplus_d q_d$ converges weakly to $\mu\boxplus \nu$ when $d$ tends to infinity, where $\mu,\nu\in \calP$ are limit laws of sequences of empirical root distribution of $p_d$ and $q_d$, respectively.
Moreover, Marcus \cite{Marcus} obtained the typical limit theorems (LLN, the central limit theorem and the Poisson's law of small numbers, etc.) for finite free additive convolution.
According to the evidence above, we can treat finite free probability as a discrete approximation theory for free probability.

In this paper, we investigate the LLN for finite free multiplicative convolution. The finite free multiplicative convolution $\boxtimes_d$ of monic polynomials $p(x)=\sum_{i=0}^d (-1)^i p_i x^{d-i}$ and $q(x)=\sum_{i=0}^d (-1)^i q_i x^{d-i}$ of degree $d$ with non-negative real roots is defined by
\begin{align*}
\big(p \boxtimes_d q\big)(x):=\sum_{i=0}^d (-1)^i\frac{p_iq_i}{\binom{d}{i}} x^{d-i}.
\end{align*}
In particular,  $p^{\boxtimes_dn}$ denotes the $n$-th power of finite free multiplicative convolution of $p$.
We formulate the LLN for roots of finite free multiplicative convolution of polynomials as the convergence of a sequence of
\begin{align}\label{eq:rootlimit}
\Big\{ \big(\lambda_{1}^{(n)}\big)^{\frac{1}{n}}  \ge \dots \ge \big(\lambda_{d}^{(n)}\big)^{\frac{1}{n}} \Big\}_{n \in \N},
\end{align}
where $\lambda_{i}^{(n)}$ is the $i$-th (non-negative) root of $p^{\boxtimes_dn}$ for a monic polynomial $p$ of degree $d$ with non-negative roots.

We define the notations for later use.
For a finite multiset $\Lambda=\{\lambda_1,\dots, \lambda_d\}$ of complex numbers, the $i$-th elementary symmetric polynomials $e_i(\Lambda)$ is denoted by
\begin{align*}
e_i(\Lambda):= \sum_{J \subset [d],\, |J|=i} \bigg( \prod_{j \in J} \lambda_j \bigg), \qquad e_0(\Lambda):=1,
\end{align*}
where $[d]:=\{1,2,\dots, d\}$.
In addition, we define
\begin{align*}
\widetilde{e}_i(\Lambda):=\frac{e_i(\Lambda)}{\binom{d}{i}}.
\end{align*}
for each $0 \le i \le d$.

We then obtain the limit theorem for the sequence \eqref{eq:rootlimit} as follows.

\begin{Theorem}[LLN for $\boxtimes_d$]\label{thm:main}
Let $p$ be a monic polynomial of degree $d$ with non-negative real roots $\Lambda$ and let us set $k=k(p)$ as the number of zeros in $\Lambda$.
Then,
\begin{align*}
\lim_{n\rightarrow\infty } \big(\lambda_{i}^{(n)}\big)^{\frac{1}{n}} = \frac{\tilde{e}_i(\Lambda)}{\tilde{e}_{i-1}(\Lambda)}
\end{align*}
for $1\le i \le d-k$, and $\lim_{n\rightarrow\infty } \big(\lambda_{i}^{(n)}\big)^{\frac{1}{n}} =0$ for $d-k+1 \le i \le d$.
\end{Theorem}

The paper consists of $4$ sections.
In Section~\ref{section2}, we introduce some concepts and preliminary results on free probability and finite free probability theories. In Section~\ref{section3}, we study the roots of finite free multiplicative convolution of polynomials and provide a proof of our main result (Theorem \ref{thm:main}).
In Section~\ref{section4}, we investigate the behavior of the empirical root distribution of polynomials obtained by the LLN for finite free multiplicative convolution when their degree tends to infinity. In the last of this section, we give a conjecture related to a connection between LLNs for $\boxtimes_d$ and $\boxtimes$ from evidence obtained by this section.

\section{Preliminaries}\label{section2}
\subsection{Free multiplicative convolution}\label{section2.1}

In this section, we introduce free multiplicative convolution and its characterization via the S-transform (see \cite{BV93} for more details). For a probability measure $\mu\neq \delta_0$ on $[0,\infty)$, we define
\begin{align*}
\psi_\mu(z):=\int_0^\infty \frac{tz}{1-tz}\ \mu({\rm d}t), \qquad z\in \mathbb{C}\setminus [0,\infty).
\end{align*}
It is known that its inverse function $\psi_\mu^{-1}$ exists in a neighborhood of $(\mu(\{0\})-1,0)$, and so we define the {\it S-transform} of $\mu$ by
\begin{align*}
S_\mu(z):=\frac{z+1}{z}\psi_\mu^{-1}(z), \qquad z\in (\mu(\{0\})-1,0).
\end{align*}
According to \cite{BV93}, for probability measures $\mu,\nu\neq \delta_0$ on $[0,\infty)$, the free multiplicative convolution $\mu\boxtimes \nu$ is characterized by
\begin{align*}
S_{\mu\boxtimes \nu}(z)=S_\mu(z)S_\nu(z),
\end{align*}
for all $z$ in the common interval where all three S-transform are defined. Note that the common interval is not empty since $\big(\mu\boxtimes \nu\big)(\{0\})=\max\{\mu(\{0\}),\nu(\{0\})\}$ (see \cite[Lemma 6.9]{BV93}).

The LLN for free multiplicative convolution of a probability measure on $[0,\infty)$ was obtained by Tucci \cite{Tucci10} and Haagerup and M\"{o}ller \cite{HM13}.

\begin{Proposition}\label{prop:HMlimit}
Let us consider $\mu\in \calP_+$. As $n\rightarrow\infty$, the sequence of $\big(\mu^{\boxtimes n}\big)^{\frac{1}{n}}$ converges weakly to the measure $\Phi(\mu)\in\calP_+$ characterized by
\begin{align*}
\Phi(\mu)&\left(\left[0,\frac{1}{S_\mu(t-1)}\right]\right)=t, \qquad t\in (\mu(\{0\}),1),\\
\Phi(\mu)&(\{0\})=\mu(\{0\}).
\end{align*}
Moreover, the support of the measure $\Phi(\mu)$ is the closure of the interval
\begin{align*}
\left(\left(\int_0^\infty t^{-1}  \mu({\rm d}t)\right)^{-1}, \int_0^\infty t \, \mu({\rm d}t)\right) \subset [0,\infty].
\end{align*}
\end{Proposition}

\begin{Example}\label{ex:HMlimit}\quad
\begin{enumerate}
\item[\rm (1)] Let ${\bf MP}$ be the Marchenko--Pastur distribution which is defined by
\begin{align*}
\mathbf{MP}({\rm d}t)=\frac{\sqrt{t(4-t)}}{2\pi t} \mathbf{1}_{(0,4)}(t)\, {\rm d}t.
\end{align*}
Then $\Phi({\bf MP})$ is the uniform distribution ${\bf U}(0,1)$ on the open interval $(0,1)$ by \cite{Ueda21}.
\item[\rm (2)] Consider $\mu=\frac{1}{2}(\delta_0+\delta_1)$. Then we have
\[
\Phi(\mu)=\frac{1}{2}\delta_0+\frac{1}{2(1-t)^2}\mathbf{1}_{(0,1/2)}(t)\,{\rm d}t,
\]
since $S_\mu(t)=(2+2t)/(1+2t)$, $t\in (-1/2,0)$ implies that $\Phi(\mu)([0,t])=2^{-1}(1-t)^{-1}$ for all $1/2 < t < 1$.
\end{enumerate}
\end{Example}

\subsection{Finite free multiplicative convolution}

In this section, we introduce some concepts and preliminary results on finite free probability that are used in the remainder of this paper; see \cite{AVP21, Marcus, MSS22} for more details.

\begin{Definition}\label{defn}
For monic polynomials $p$ and $q$ of degree $d$ which have only non-negative real roots:
\begin{align*}
p(x)=\sum_{i=0}^d (-1)^i p_ix^{d-i}, \qquad \text{ and } \qquad q(x)=\sum_{i=0}^d (-1)^i q_ix^{d-i},
\end{align*}
the {\it finite free multiplicative convolution} $p\boxtimes_d q$ is defined by
\begin{align*}
\big(p\boxtimes_d q\big) (x):= \sum_{i=0}^d(-1)^i \frac{p_iq_i}{\binom{d}{i}} x^{d-i}.
\end{align*}
\end{Definition}

In \cite[Theorem 1.5]{MSS22}, the finite free multiplicative convolution $\boxtimes_d$ can be realized as a characteristic polynomial of a product of positive definite matrices. That is, if $\chi_{A}$ and $\chi_{B}$ are characteristic polynomials of $d\times d$ positive definite matrices $A$ and $B$, respectively, then
\begin{align*}
\big(\chi_{A} \boxtimes_d \chi_{B}\big)(x)=\mathbb{E}_Q \det[x I - AQBQ^*],
\end{align*}
where the expectation is taken over unitary matrices $Q$ distributed uniformly on the unitary group in dimension $d$.
Moreover, if $p$ and $q$ have only non-negative real roots, then so is $p\boxtimes_d q$ (see \cite[Theorem 1.6]{MSS22}).

There is the following asymptotic relation between finite free multiplicative convolution and free multiplicative convolution by \cite[Theorem 1.4]{AVP21}.
Let us consider $p_d$ and $q_d$ as real-rooted monic polynomials of degree $d$ in which $p_d$ has only non-negative real roots. Assume the empirical root distributions of $p_d$ and $q_d$ converge weakly to $\mu\in \calP_{+,c}$ and $\nu\in \calP_c$ as $d\rightarrow\infty$, respectively. Then the empirical root distribution of $p_d \boxtimes_d q_d$ converges weakly to $\mu\boxtimes \nu$ as $d\rightarrow\infty$.

\section{Main result}\label{section3}

In this section, we provide the LLN for finite free multiplicative convolution. First, we calculate the $n$-th power of finite free multiplicative convolution of polynomials which have only non-negative real roots. Let $\Lambda^{(n)}$ be the set of roots of $p^{\boxtimes_d n}$ for $n \ge1$ and a monic polynomial $p$ of degree $d$ with non-negative real roots.
We put $\Lambda := \Lambda^{(1)}$, for short.

\begin{Lemma}\label{lem:free_multiplicative_polynomial}
Let $p$ be a monic polynomial of degree $d$ with non-negative real roots $\Lambda$. Then we have
\begin{align} \label{eq:symm_root}
\widetilde{e}_{i}\big(\Lambda^{(n)}\big)= \widetilde{e}_{i}(\Lambda)^n, \qquad 0\le i \le d.
\end{align}
In particular, the number of zeros in $\Lambda$ is the same as the one in $\Lambda^{(n)}$.
\end{Lemma}

\begin{proof}
Note that
\begin{align*}
p(x)= \sum_{i=0}^{d} (-1)^i \binom{d}{i} \widetilde{e}_i(\Lambda) x^{d-i},
\end{align*}
then by Definition \ref{defn}
\begin{align*}
p^{\boxtimes_d n}(x)= \sum_{i=0}^d (-1)^i \binom{d}{i} \widetilde{e}_i(\Lambda)^{n} x^{d-i}.
\end{align*}
This is equivalent to \eqref{eq:symm_root}.
The rest is because the number of zeros in $\Lambda$ is equal to $k$ if and only if
\begin{align*}
e_{d-k}(\Lambda) > 0 \quad \text{and} \quad e_{d-k+1}(\Lambda) = 0,
\end{align*}
where we understand $e_{d+1}(\Lambda) =0$.
\end{proof}

Due to the relation \eqref{eq:symm_root}, we obtain the following LLN for roots of finite free multiplicative convolution of polynomials.

\begin{Theorem}\label{thm:LLN}
Consider a monic polynomial $p$ of degree $d$ with non-negative real roots $\Lambda$ and let $k=k(p)$ be the number of zeros in $\Lambda$.
Let $\Lambda^{(n)}:=\big\{\lambda_{1}^{(n)} \ge \lambda_{2}^{(n)} \ge\dots \ge \lambda_{d}^{(n)} \big\} $ be the set of non-negative real roots of $p^{\boxtimes_d n}$.
Then we get
\begin{align*}
\lim_{n\rightarrow\infty}\big(\lambda_{i}^{(n)}\big)^{\frac{1}{n}} =\frac{\widetilde{e}_i(\Lambda)}{\widetilde{e}_{i-1}(\Lambda)}, \qquad 1 \le i \le d-k.
\end{align*}
\end{Theorem}

\begin{Remark}
Note that $\lambda_{i}^{(n)} = 0$ for $d-k+1 \le i \le d$ by Lemma \ref{lem:free_multiplicative_polynomial}.
\end{Remark}

\begin{proof}
For $1\le i \le d-k$, the equation \eqref{eq:symm_root} implies that
\begin{align}\label{eq:Lambda_S}
\frac{\widetilde{e}_{i}\big(\Lambda^{(n)}\big)}{\widetilde{e}_{i-1}\big(\Lambda^{(n)}\big)}= \left( \frac{\widetilde{e}_{i}(\Lambda)}{\widetilde{e}_{i-1}(\Lambda)}\right)^n.
\end{align}

For $i=1$, the equation \eqref{eq:Lambda_S} implies that
\begin{align*}
\frac{\lambda_{1}^{(n)}+\dots + \lambda_{d}^{(n)}}{d}=\widetilde{e}_1\big(\Lambda^{(n)}\big)=\widetilde{e}_1(\Lambda)^n.
\end{align*}
Since
\begin{align*}
\frac{\lambda_{1}^{(n)}}{d} \le \frac{\lambda_{1}^{(n)}+\dots + \lambda_{d}^{(n)}}{d} \le \lambda_{1}^{(n)},
\end{align*}
we obtain
\begin{align*}
\widetilde{e}_1(\Lambda) \le \big(\lambda_{1}^{(n)}\big)^{\frac{1}{n}} \le d^{\frac{1}{n}}\widetilde{e}_1(\Lambda),
\end{align*}
and therefore {$\big(\lambda_{1}^{(n)}\big)^{\frac{1}{n}} \rightarrow \widetilde{e}_1(\Lambda)$ as $n\rightarrow\infty$.}

For {$2\le i \le d-k$}, we have
\begin{align*}
\frac{e_{i}\big(\Lambda^{(n)}\big)}{e_{i-1}\big(\Lambda^{(n)}\big)} &= \frac{\sum_{J \subset [d],\ |J|=i} \left( \prod_{j \in J} \lambda_{j}^{(n)} \right)}{ \sum_{J \subset [d],\ |J|=i-1} \left( \prod_{j \in J} \lambda_{j}^{(n)} \right)}\ge \frac{\sum_{J \subset [d],\ |J|=i} \left( \prod_{j \in J} \lambda_{j}^{(n)} \right)}{\binom{d}{i-1} \lambda_{1}^{(n)}\lambda_{2}^{(n)}\cdots \lambda_{i-1}^{(n)}}\\
&\ge \frac{\lambda_{1}^{(n)}\lambda_{2}^{(n)}\cdots \lambda_{i}^{(n)}}{\binom{d}{i-1} \lambda_{1}^{(n)}\lambda_{2}^{(n)}\cdots \lambda_{i-1}^{(n)}}=\binom{d}{i-1}^{-1} \lambda_{i}^{(n)}.
\end{align*}
Similar to the estimation above, we obtain
\begin{align*}
\frac{e_{i}\big(\Lambda^{(n)}\big)}{e_{i-1}\big(\Lambda^{(n)}\big)}  &\le \frac{\binom{d}{i}\lambda_{1}^{(n)}\lambda_{2}^{(n)}\cdots \lambda_{i}^{(n)} }{\lambda_{1}^{(n)}\lambda_{2}^{(n)}\cdots \lambda_{i-1}^{(n)}} = \binom{d}{i} \lambda_{i}^{(n)}.
\end{align*}
Consequently, we have
\begin{align*}
\binom{d}{i}^{-1} \frac{e_i\big(\Lambda^{(n)}\big)}{e_{i-1}\big(\Lambda^{(n)}\big)} \le \lambda_{i}^{(n)} \le \binom{d}{i-1} \frac{e_i\big(\Lambda^{(n)}\big)}{e_{i-1}\big(\Lambda^{(n)}\big)},
\end{align*}
and therefore
\begin{align*}
\binom{d}{{i-1}}^{-1} \left(\frac{\widetilde{e}_i(\Lambda)}{\widetilde{e}_{i-1}(\Lambda)} \right)^n \le \lambda_{i}^{(n)} \le \binom{d}{{i}} \left(\frac{\widetilde{e}_i(\Lambda)}{\widetilde{e}_{i-1}(\Lambda)} \right)^n
\end{align*}
by using \eqref{eq:Lambda_S}.
Taking the $n$-th root of each value in the above inequality, we get
\begin{align*}
\binom{d}{{i-1}}^{-\frac{1}{n}} \frac{\widetilde{e}_i(\Lambda)}{\widetilde{e}_{i-1}(\Lambda)} \le \big(\lambda_{i}^{(n)}\big)^{\frac{1}{n}} \le \binom{d}{{i}}^{\frac{1}{n}} \frac{\widetilde{e}_i(\Lambda)}{\widetilde{e}_{i-1}(\Lambda)}.
\end{align*}
Hence, we obtain $\big(\lambda_{i}^{(n)}\big)^{\frac{1}{n}} \rightarrow   {\widetilde{e}_i(\Lambda)}/{\widetilde{e}_{i-1}(\Lambda)}$ as $n\rightarrow\infty$.
\end{proof}

For a positive number $\alpha > 0$ and a monic polynomial $p(x)=\prod_{i=1}^d (x-\lambda_{i})$ with non-negative real roots, we define
\[
p^{(\alpha)}(x):=\prod_{i=1}^d \big(x-\lambda_{i}^\alpha\big).
\]

\begin{Remark}\label{rem:LLN_poly}
According to Theorem \ref{thm:LLN}, if $p$ is a monic polynomial of degree $d$ with non-negative real roots $\Lambda$ and $k$ is the number of zeros in $\Lambda$, then
\begin{align*}
\lim_{n\rightarrow \infty} \big(p^{\boxtimes_d n}\big)^{\left(\frac{1}{n}\right)}(x)=x^{k}\prod_{i=1}^{d-k} \left(x-  \frac{\widetilde{e}_i(\Lambda)}{\widetilde{e}_{i-1}(\Lambda)} \right).
\end{align*}
Thus, the LLN for finite free multiplicative convolution of polynomials is established.
\end{Remark}

\begin{Remark}
Let $p$ be a monic polynomial of degree $d$ with non-negative real roots $\Lambda=\{\lambda_{1} \ge \dots  \ge \lambda_d\}$, $k$ the number of zeros in $\Lambda$, and $\lambda_{i}^{(n)}$ the $i$-th real root of $p^{\boxtimes_d n}$ for $1\le i \le d$. By Newton's inequality (see, e.g., \cite{HLP34}), we have
\begin{align}\label{eq:Newton}
\frac{\widetilde{e}_i(\Lambda)}{\widetilde{e}_{i-1}(\Lambda)}\ge \frac{\widetilde{e}_{i+1}(\Lambda)}{\widetilde{e}_i(\Lambda)}, \qquad 1\le i \le d-k-1,
\end{align}
where the equality holds if and only if $\lambda_{1}= \dots = \lambda_{d}$. However, the inequality \eqref{eq:Newton} can be directly proven by Theorem \ref{thm:LLN} due to $\lambda_{i}^{(n)} \ge \lambda_{i+1}^{(n)}$.

Consequently, we find the following remarkable phenomenon; except for trivial cases, the limit roots of $\big(p^{\boxtimes_d n}\big)^{\left(\frac{1}{n}\right)}$, not being zero, are all distinct.
\end{Remark}

In particular, we apply Theorem \ref{thm:LLN} to the (renormalized) Laguerre polynomial and a polynomial with two real roots.

\begin{Example}[case of the Laguerre polynomial]\label{ex:Laguerre}
Consider $d\ge 1$. We define
\begin{align*}
p(x):=d! (-d)^{-d} L_{0,d}(dx),
\end{align*}
where $L_{\alpha,d}$ is the {\it Laguerre polynomial} which is defined by
\begin{align*}
L_{\alpha,d}(x):=\sum_{i=0}^d \frac{(-x)^i}{i!} \binom{d+\alpha}{d-i}.
\end{align*}

Let $\Lambda$ be the set of positive real roots of $p$. Note that $p$ has no zero roots since $p(0)=d! (-d)^{-d} \neq 0$. Computing the polynomial $p$, we have
\begin{align*}
p(x) &=d! (-d)^{-d} \sum_{i=0}^d \frac{(-dx)^i}{i!} \binom{d}{d-i}\\
&=x^d+\sum_{j=1}^d (-1)^j \binom{d}{j} \left(\frac{d}{d}\cdot \frac{d-1}{d} \cdots \frac{d-j+1}{d} \right) x^{d-j},
\end{align*}
where the last equality holds by changing variable to $j=d-i$. This implies that
\begin{align*}
\widetilde{e}_j(\Lambda)=\frac{d}{d}\cdot \frac{d-1}{d} \cdots \frac{d-j+1}{d}, \qquad 1 \le j \le d.
\end{align*}

Suppose that $\lambda_{1}^{(n)}\ge \dots \ge \lambda_{d}^{(n)}$ are non-negative real roots of $p^{\boxtimes_d n}$. By Theorem \ref{thm:LLN}, we obtain
\begin{align*}
\lim_{n\rightarrow\infty}\big(\lambda_{i}^{(n)}\big)^{\frac{1}{n}} &= \frac{\widetilde{e}_i(\Lambda)}{\widetilde{e}_{i-1}(\Lambda)}=\frac{\frac{d}{d}\cdot \frac{d-1}{d} \cdots \frac{d-i+1}{d}}{\frac{d}{d}\cdot \frac{d-1}{d} \cdots \frac{d-i+2}{d}}=\frac{d-i+1}{d}
\end{align*}
for $1\le i \le d$, where note that $\tilde{e}_0(\Lambda)=1$.
\end{Example}

\begin{Example}[case of a polynomial with two roots]\label{ex:two_root}
Given $d\ge 1$, consider the following monic polynomial $p$ of $2d$ degree:
\begin{align*}
p(x)=x^d(x-1)^d, \qquad d \ge1,
\end{align*}
and put $\Lambda=\{\underbrace{1,\dots, 1}_{d \text{ times}} , \underbrace{0, \dots, 0}_{d \text{ times}}\}$ as the set of roots of $p$. Then we get
\begin{align*}
\widetilde{e}_j(\Lambda)=\frac{\binom{d}{j}}{\binom{2d}{j}}, \qquad 0 \le j \le d,
\end{align*}
and $\tilde{e_j}(\Lambda)=0$ for $d+1\le j \le 2d$. Suppose that $\lambda_{1}^{(n)} \ge \dots \ge\lambda_{d}^{(n)}$ are positive real roots of $p^{\boxtimes_d n}$. By Theorem \ref{thm:LLN}, we have
\begin{align*}
\lim_{n\rightarrow\infty} \big(\lambda_{i}^{(n)}\big)^{\frac{1}{n}}=\frac{d-i+1}{2d-i+1},
\end{align*}
for all $1\le i \le d$.
\end{Example}

At the end of this section, we mention the rate of convergence in the LLN for roots of finite free multiplicative convolution of polynomials.

\begin{Remark}
According to a proof of Theorem \ref{thm:LLN}, it is easy to see that
\begin{align*}
\log \big(\lambda_i^{(n)}\big)^{\frac{1}{n}} = \log \left(\frac{\widetilde{e}_i(\Lambda)}{\widetilde{e}_{i-1}(\Lambda)}\right)
+ O\left(\frac{1}{n}\right),
\end{align*}
as $n\rightarrow\infty$.

We demonstrate an example in which the rate of convergence is of order $1/n$ and it is optimal. Consider $d=2$ in Example \ref{ex:Laguerre}, that is, $p(x)=x^2-2x+2^{-1}$. As a consequent result of a proof of Lemma \ref{lem:free_multiplicative_polynomial}, we obtain $p^{\boxtimes_2 n} (x)=x^2-2x+2^{-n}$ for $n\in\N$. Hence (positive) real roots of $p^{\boxtimes_2 n}$ are given by
\[
\lambda_1^{(n)}=1+\sqrt{1-2^{-n}}, \qquad \lambda_2^{(n)}=1-\sqrt{1-2^{-n}}.
\]
It is easy to see that $\lim_{n\rightarrow\infty} \big(\lambda_1^{(n)}\big)^{\frac{1}{n}}=1$, $\lim_{n\rightarrow\infty} \big(\lambda_2^{(n)}\big)^{\frac{1}{n}} =1/2$ and also
\begin{align*}
n \left( \log\big(\lambda_1^{(n)}\big)^{\frac{1}{n}} - \log 1 \right) \rightarrow \log 2, \qquad
n\left(\log \big(\lambda_2^{(n)}\big)^{\frac{1}{n}} - \log \frac{1}{2} \right) \rightarrow -\log2,
\end{align*}
as $n\rightarrow\infty$. Consequently, the order $1/n$ is optimal.
\end{Remark}

\section{Relation to the LLN for free multiplicative convolution}\label{section4}

In this section, given a monic polynomial of degree $d$, we investigate how the empirical root distributions of their limit polynomial obtained by Theorem \ref{thm:LLN} (or Remark \ref{rem:LLN_poly}) converge weakly as $d\rightarrow\infty$. For the reason above, we emphasize their degree as follows.

Let $p_d$ be a monic polynomial of degree $d$ with non-negative real roots $\Lambda_d = \{\lambda_{1,d}\ge \dots \ge \lambda_{d,d}\}$ and let $k_d$ be the number of zeros in $\Lambda_d$.
Denote by $R_i(\Lambda_d)$ the limit roots of $\big(p_d^{\boxtimes_d n}\big)^{\left(\frac{1}{n}\right)}$ as $n\rightarrow\infty$ for $1\le i \le d$, that is,
\[
R_i(\Lambda_d) = \frac{\tilde{e}_i(\Lambda_d)}{\tilde{e}_{i-1}(\Lambda_d)},
\]
as provided in Theorem \ref{thm:LLN}.
In the following, we investigate relationships between the empirical root distributions:
\[
\mu_d:=\frac{1}{d}\sum_{i=1}^d \delta_{\lambda_{i,d}}, \qquad \text{and} \qquad \nu_d:=\frac{1}{d}\sum_{i=1}^d \delta_{R_{i}(\Lambda_d)}.
\]

\begin{Lemma}\label{lem:log}
Let $p_d$ be a monic polynomial of degree $d$ with non-negative real roots $\Lambda_d$. Assume that $k_d=0$ $($equivalently, $\lambda_{d,d}>0$ or $R_d(\Lambda_d)>0$$)$. Then we have
\begin{align*}
\int_0^\infty (\log t ) \, \mu_d ({\rm d}t) =\int_0^\infty (\log t) \, \nu_d ({\rm d}t).
\end{align*}
\end{Lemma}

\begin{proof}
Note that the integrals $\int_0^\infty (\log t) \, \mu_d({\rm d}t)$ and $\int_0^\infty (\log t) \, \nu_d({\rm d}t)$ are finite since $k_d=0$. A~direct computation shows that
\begin{align*}
\int_0^\infty (\log t) \, \nu_d({\rm d}t) &= \frac{1}{d}\sum_{i=1}^d \log R_i(\Lambda_d)\\
&=\frac{1}{d}\sum_{i=1}^d \log \frac{\widetilde{e}_i(\Lambda_d)}{\widetilde{e}_{i-1}(\Lambda_d)} \qquad \text{(by Theorem \ref{thm:LLN})}\\
&=\frac{1}{d}\log \widetilde{e}_d(\Lambda_d) \qquad \text{(since $\widetilde{e}_0(\Lambda_d)=1$)}\\
&=\frac{1}{d}\log \prod_{i=1}^d \lambda_{i,d}\\
&=\frac{1}{d}\sum_{i=1}^d \log \lambda_{i,d}=\int_0^\infty (\log t)\, \mu_d({\rm d}t).\tag*{\qed}
\end{align*}\renewcommand{\qed}{}
\end{proof}

We study how the empirical root distributions $\nu_d=\frac{1}{d}\sum_{i=1}^d \delta_{R_{i}(\Lambda_d)}$ behave as $d\rightarrow\infty$ when
$\mu_d=\frac{1}{d}\sum_{i=1}^d \delta_{\lambda_{i,d}}$ converges weakly to some probability measure on $[0,\infty)$.

\begin{Proposition}\label{prop:R}
Let $p_d$ be a monic polynomial of degree $d$ with non-negative real roots $\Lambda_d$. Assume that there exist $\mu\in \calP_{+,c}$ and a compact set $K$ in $[0,\infty)$, such that the measures $\mu_d$ and~$\mu$ are supported on $K$ for all $d\ge 1$, and such that $\mu_d \xrightarrow{w} \mu$ as $d\rightarrow\infty$.
Then we obtain
\[
R_1(\Lambda_d)\rightarrow \int_0^\infty t \, \mu({\rm d}t)
\]
as $d\rightarrow\infty$. In addition, if $0\notin K$, then
it satisfies that
\[
R_{d}(\Lambda_d)\rightarrow \left(\int_0^\infty t^{-1} \mu({\rm d}t)\right)^{-1} \qquad \text{and} \qquad
\int_0^\infty (\log t) \nu_d({\rm d}t) \rightarrow \int_0^\infty (\log t) \Phi(\mu)({\rm d}t)
\]
as $d\rightarrow\infty$, where $\Phi(\mu)$ is defined in Proposition~$\ref{prop:HMlimit}$.
\end{Proposition}

\begin{proof}
By the assumptions and Theorem \ref{thm:LLN}, we get
\begin{align*}
R_1(\Lambda_d)=\widetilde{e}_{1}(\Lambda_d)=\frac{1}{d} \sum_{i=1}^d \lambda_{i,d}=\int_0^\infty t\, \mu_d({\rm d}t)\rightarrow \int_0^\infty t \, \mu({\rm d}t)
\end{align*}
as $d\rightarrow\infty$.

Moreover, we assume that $0\notin K$ in the following.
Note that the functions
$t\mapsto t^{-1}$ and $t\mapsto \log t$ are  bounded and continuous on
$K$.
We then obtain
\begin{align*}
R_{d}(\Lambda_d)=\frac{\widetilde{e}_{d}(\Lambda_d)}{\widetilde{e}_{d-1} (\Lambda_d)}=\left(\frac{1}{d} \sum_{i=1}^{d} \lambda_{i,d}^{-1} \right)^{-1}=\left(\int_0^\infty t^{-1}\mu_d({\rm d}t)\right)^{-1}\rightarrow \left(\int_0^\infty t^{-1}\mu({\rm d}t)\right)^{-1}
\end{align*}
as $d\rightarrow\infty$.
It follows that $k_d=0$ from $0\notin K$.
By Lemma \ref{lem:log}, we obtain
\begin{align*}
\int_0^\infty (\log t) \nu_d({\rm d}t)=\int_0^\infty (\log t) \mu_d({\rm d}t)\rightarrow \int_0^\infty (\log t) \mu({\rm d}t).
\end{align*}
According to \cite[Proposition 1]{HM13}, the last integral equals to $\int_0^\infty (\log t) \Phi(\mu)({\rm d}t)$, and therefore we get the convergence.
\end{proof}

We give examples of the weak limit laws of empirical root distributions $\frac{1}{d}\sum_{i=1}^d \delta_{R_i(\Lambda_d)}$ as $d\rightarrow\infty$.

\begin{Example}\label{ex:converges}\quad
\begin{enumerate}\itemsep=0pt
\item[\rm (1)] In Example \ref{ex:Laguerre}, it was shown that $R_i(\Lambda_d)=\frac{d-i+1}{d}$ when we consider
\[
p_d(x)=d!(-d)^{-d}L_{0,d}(dx)
\] with non-negative real roots $\Lambda_d$ for each $d\ge 1$. It is easy to see that
\begin{align*}
\frac{1}{d}\sum_{i=1}^d \delta_{\frac{d-i+1}{d}}\xrightarrow{w} {\bf U}(0,1)=\Phi({\bf MP}),
\end{align*}
as $d\rightarrow\infty$, where the last equality holds due to Example \ref{ex:HMlimit}\,(1).

\item[\rm (2)] In Example \ref{ex:two_root}, we obtained that $R_i(\Lambda_{2d})=\frac{d-i+1}{2d-i+1}$ for $1\le i \le d$ when $p_d(x)=x^d(x-1)^d$ with non-negative real roots $\Lambda_{2d}$. For any bounded continuous functions $f$ on $[0,\infty)$, we get
\begin{align*}
\int_0^\infty f(t) \left[\frac{1}{2} \delta_0 + \frac{1}{2d} \sum_{i=1}^d \delta_{\frac{d-i+1}{2d-i+1}}\right] ({\rm d}t) &=\frac{1}{2} f(0) + \frac{1}{2d}\sum_{i=1}^d f\left(\frac{d-i+1}{2d-i+1}\right)\\
&=\frac{1}{2} f(0) + \frac{1}{2d}\sum_{\ell=1}^d f\left(\frac{\ell}{d+\ell}\right)\\
&=\frac{1}{2}f(0)+\frac{1}{2d} \sum_{\ell=1}^d f\left(\frac{\frac{\ell}{d}}{1+\frac{\ell}{d}} \right)\\
&\rightarrow \frac{1}{2} f(0) +\frac{1}{2}\int_0^1 f\left( \frac{t}{1+t}\right){\rm d}t\\
&=\frac{1}{2} f(0)+\frac{1}{2} \int_0^{1/2}  \frac{f(u)}{(1-u)^2}{\rm d}u,
\end{align*}
where the last equality holds by changing variable to $u=t/(1+t)$, and therefore
\[
\frac{1}{2} \delta_0 + \frac{1}{2d} \sum_{i=1}^d \delta_{\frac{d-i+1}{2d-i+1}} \xrightarrow{w} \frac{1}{2}\delta_0+ \frac{1}{2(1-t)^2}\mathbf{1}_{(0,1/2)}(t){\rm d}t= \Phi\left(\frac{1}{2}(\delta_0+\delta_1)\right),
\]
as $d\rightarrow\infty$, where the last equality holds due to Example \ref{ex:HMlimit}\,(2).
\end{enumerate}
\end{Example}

According to Proposition \ref{prop:R} and Example \ref{ex:converges}, it is natural to conjecture the following statement.

\begin{Conjecture}
Let $p_{d}$ be a monic polynomial of degree $d$ with non-negative real roots $\Lambda_d=\{\lambda_{1,d}\ge \dots \ge \lambda_{d,d}\}$. Let us further consider $\mu \in \calP_{+,c}$. Assume that the empirical root distributions of $p_d$, that is,
$\frac{1}{d}\sum_{i=1}^{d} \delta_{\lambda_{i,d}}$ converge weakly to $\mu$ as $d\rightarrow\infty$. Then we obtain%
\begin{align*}
\frac{1}{d}\sum_{i=1}^{d} \delta_{R_i(\Lambda_d)} \xrightarrow{w} \Phi(\mu)
\end{align*}
as $d\rightarrow\infty$.
\end{Conjecture}

\subsection*{Acknowledgements}

This work was supported by JSPS Open Partnership Joint Research Projects JPJSBP120209921 and Bilateral Joint Research Projects (JSPS-MEAE-MESRI, JPJSBP120203202). Moreover, K.F. was supported by the Hokkaido University Ambitious Doctoral Fellowship (Information Science and AI).
Y.U. was supported by JSPS Grant-in-Aid for Scientific Research (B) 19H01791 and also JSPS Grant-in-Aid for Young Scientists 17H04823 and 22K13925.
The both authors would like to thank Nikhil Srivastava (UC Berkeley) for giving us a three-day lecture on finite free probability theory at Kyoto University.
The both authors also thank Beno\^it Collins (Kyoto University) and Noriyoshi Sakuma (Nagoya City University) for introducing the discussion.
Thanks to the lectures, the authors could start to study the LLN for finite free multiplicative convolution.
Furthermore, the authors would also like to thank the referees who gave useful comments to improve this paper.

\pdfbookmark[1]{References}{ref}
\LastPageEnding

\end{document}